\documentclass[12pt]{amsart}
\usepackage{amssymb,amsmath,amsthm,mathrsfs,multirow,xcolor,framed,url}
\usepackage[pdftex,
         pdftitle={Hurwitz class numbers and mock theta functions},
         pdfproducer={Latex with hyperref},
         pdfcreator={pdflatex}]{hyperref}
\newtheorem{theorem}{Theorem}[section]
\newtheorem{lemma}[theorem]{Lemma}
\newtheorem{cor}[theorem]{Corollary}

\newtheorem{prop}[theorem]{Proposition}

\theoremstyle{definition}

\theoremstyle{remark}

\numberwithin{equation}{section}

\newcommand\nutwid{\overset {\text{\lower 3pt\hbox{$\sim$}}}\nu}

\allowdisplaybreaks
\newcommand{\beqs}{\begin{equation*}}
\newcommand{\eeqs}{\end{equation*}}
\newcommand{\beq}{\begin{equation}}
\newcommand{\eeq}{\end{equation}}
\DeclareMathOperator{\IM}{Im}

\makeatletter
\@namedef{subjclassname@2020}{\textup{2020} Mathematics Subject Classification}
\makeatother
\begin{document}

\title[Hurwitz class numbers and mock theta functions]{Generating functions of the Hurwitz class numbers associated with certain mock theta functions}

\author{Dandan Chen}
\address{Department of Mathematics, Shanghai University, Shanghai, People’s Republic of China}
\email{mathcdd@shu.edu.cn}
\author{Rong Chen}
\address{School of Mathematical Sciences, Tongji University, Shanghai, People’s Republic of China}
\email{rongchen20@tongji.edu.cn}



\begin{abstract}
We find Hecke-Rogers type series representations of generating functions of the Hurwitz class numbers which are similar to certain mock theta functions. We also prove two combinatorial interpretations of Hurwitz class numbers which appeared on OEIS (see A238872 and  A321440).
\end{abstract}

\maketitle

\section{Introduction}

Following \cite[Section 5.3]{Co-93}, we define the Hurwitz class number $H(N)$, where $N$ is a non-negative integer, as follows:

(1)If $N\equiv 1,2 \pmod 4$ then $H(N)=0$.

(2)If $N=0$ then $H(0)=-1/12$.

(3)If $N>0$, $N\equiv 0,3 \pmod 4$, then $H(N)$ is the class number of positive definite binary quadratic forms of discriminant $-N$, with those classes that contain a multiple of $x^2+y^2$ or $x^2+xy+y^2$ counted with weight $1/2$ or $1/3$, respectively.

For the modular behaviors, we use Kronecker's notation \cite{Kr-60} $F(n)$ satisfying $F(8n+7)=H(8n+7)$ and $F(8n+3)=3H(8n+3)$, see also \cite{Wa-35} where denote $H(n)$ as $F_1(n)$. The generating functions of $F(an+b)$ are defined as
$$
\mathscr{F}_{a,b}(q):=\sum_{n=0}^{\infty}F(an+b)q^n.
$$

In 1935 Watson \cite{Wa-35} studied the transformations of such generating functions of $H(n)$  by using Mordell integrals. Hirzebruch and Zagier \cite[Section 2.2]{Hi-Za-76} discussed the relation between the generating functions of $H(n)$ and certain weight $3/2$ modular forms. On the other hand, Zwegers \cite{Zw-02} found that mock theta functions are related to weight $1/2$ modular forms. In this paper, we find some similar results between the generating functions of Hurwitz class numbers and certain mock theta functions.

Mock theta functions were first mentioned by Ramanujan in his last letter to Hardy in 1920 as Eulerian forms. To study the modular behaviors, many authors such as Watson \cite{Wa-36,Wa-37}, Andrews \cite{An-86}, Andrews and Hickerson \cite{An-Hi-91}, Berndt and Chan \cite{Be-Ch-07}, Choi \cite{Ch-99,Ch-00,Ch-02,Ch-07}, Garvan \cite{Ga-19}, Gordon and McIntosh \cite{Go-Mc-00}, Hickerson \cite{Hi-88} and Zwegers \cite{Zw-09} find the Appell-Lerch series and Hecke-Rogers series of mock theta functions. In a recent work, Hickerson and Mortenson \cite{Hi-Mo-14} used the ``m-block" to list the Appell-Lerch series of all the mock theta functions.

Recently,  the second author and Garvan \cite{Ch-Ga} provided some of  modulo 4 congruences  between the Hurwitz class numbers and the coefficients of certain mock theta functions. For example \cite[Lemma 3.1]{Ch-Ga}
$$
N_A(n)\equiv (-1)^{n+1}H(8n-1) \pmod4,
$$
where $N_A(n)$ are coefficients of the second order mock theta function $A(q)$ (see \cite{Mc-07})
$$
A(q):=\sum_{n=0}^{\infty}\frac{q^{(n+1)^2}(-q;q^2)_n}{(q;q^2)^2_{n+1}}.
$$
We always assume that $|q|<1$ and use the standard $q$-series notations:
\begin{align*}
&(a;q)_0:=1, \quad (a;q)_n:=\prod_{k=0}^{n-1}(1-aq^k), \quad (a;q)_\infty:=\prod_{k=0}^\infty(1-aq^k),\\
&(a_1,a_2,\cdots,a_m;q)_n=(a_1;q)_n(a_2;q)_n\cdots(a_m;q)_n, \quad n\in\mathbb{N}\cup\{\infty\}.
\end{align*}
To find further relation, we find the Hecke-Rogers series of some of $\mathscr{F}_{a,b}(q)$ by following one of the methods of studying mock theta functions. Something amazing is that the Hecke-Rogers series we find are quite similar to certain mock theta functions.

Cui, Gu and Hao \cite[Eq. (2.24)]{Cu-Gu-Ha-18} showed the Hecke-Rogers series of $A(q)$ and we rewrite it as
\beq
\label{Aq0}
\frac{J_1^2}{J_2}A(q)=\sum_{1\leq j\leq |n|}sg(n)(-1)^{n-1}q^{2n^2-n-j^2+j},
\eeq
where, as usual, $sg(n)=1$ if $n\geq 0$ and $sg(n)=-1$ if $n<0$ and
$$
J_k:=(q^k;q^k)_\infty.
$$
For $\mathscr{F}_{8,-1}(q)$, there is a similar identity
\beq
\label{H80}
\frac{J_1^2}{J_2}\mathscr{F}_{8,-1}(q)=\sum_{1\leq j\leq |n|}sg(n)(-1)^{j-1}(2j-1)q^{2n^2-n-j^2+j},
\eeq
which appeared in \cite[P.376]{Hu-07}. \eqref{H80} is quite similar to \eqref{Aq0} and we find a $z$-analog of these two identities.

\begin{theorem}
\beq
\label{H8z0}
(zq,q/z,q^2;q^2)_\infty F_8(z,q)=\sum_{1\leq j\leq |n|}sg(n)(-1)^{j-1}q^{2n^2-n-j^2+j}\cdot \frac{z^{1-j}-z^j}{1-z},
\eeq
where
\beq
\label{H8zd0}
F_8(z,q)=\sum_{n=0}^{\infty}\frac{(-1)^n(q;q^2)_nq^{(n+1)^2}}{(zq,q/z;q^2)_{n+1}}.
\eeq
\end{theorem}

Our proof is based on the formulas of generating functions of Hurwitz class numbers by Alfes, Bringmann and Lovejoy \cite{Al-Br-Lo-11}, then we use the formulas of Appell-Lerch series by Hickerson and Mortenson \cite{Hi-Mo-14}. It showed that $F_8(1,q)=\mathscr{F}_{8,-1}(q)$ as a special case of $N^0(a,b;z;q)$ in \cite[Eq. (1.5)]{Al-Br-Lo-11}. Setting $(z,q)=(1,q)$ and $(z,q)=(-1,-q)$, noting that $F_8(-1,-q)=-A(q)$, \eqref{H8z0} yields \eqref{H80} and \eqref{Aq0}, respectively.

We obtain a similar result for the eighth order mock theta function $V_1(q)$ and $\mathscr{F}_{4,-1}(q)$, where (see \cite{Go-Mc-00})
$$
V_1(q):=\sum_{n=0}^{\infty}\frac{q^{(n+1)^2}(-q;q^2)_n}{(q;q^2)_{n+1}}.
$$

\begin{theorem}
For $z\in \mathbb{C^*}$ and $z\neq 1$, we have
\beq
\label{4th2id}
(-zq,-1/z,q;q)_\infty F_4(z,-q)=\sum_{1\leq j\leq n}q^{n^2-j(j-1)/2}\cdot \frac{z^n-z^{-n}}{1-z},
\eeq
where
$$
F_4(z,q)=\sum_{n=0}^{\infty}\frac{(-1)^n(q;-q)_{2n}q^{n+1}}{(zq,q/z;q^2)_{n+1}}.
$$
\end{theorem}

We have that $F_4(1,q)=\mathscr{F}_{4,-1}(q)$ and $F_4(i,-q)=-V_1(q)$.

We also find and prove Hecke-Rogers series for $\mathscr{F}_{a,-1}(q),(a=12,24)$ which are similar to two sixth order mock theta functions $\sigma(q)$ (see \cite{An-Hi-91}) and $\phi_{-}(q)$ (see \cite{Be-Ch-07})
$$
\sigma(q):=\sum_{n=0}^{\infty}\frac{q^{(n+1)(n+2)/2}(-q;q)_n}{(q;q^2)_{n+1}},
$$
$$
\phi_{-}(q):=\sum_{n=1}^{\infty}\frac{q^n(-q;q)_{2n-1}}{(q;q^2)_n}.
$$
But we cannot find $z$-analog of Hecke-Rogers series for $\mathscr{F}_{a,-1}(q),(a=12,24)$.

In Section 2, we recall some useful formulas for Appell-Lerch series, Garvan's MAPLE packages and derivative of certain theta functions. In Section 3, we find further connections between Hurwitz class numbers and certain mock theta fucntions. In the final section, we prove combinatorial interpretations of $F(4n-1)$ and $H(8n-1)$ which appeared on OEIS (see A238872 and  A321440).

\section{Preparation}

\subsection{Appell-Lerch series}

Following \cite{Hi-Mo-14}, we define
$$
j(x;q):=(x,q/x,q;q)_\infty=\sum_{n=-\infty}^{\infty}(-1)^nq^{n(n-1)/2}x^n,
$$
 and the Appell-Lerch series
$$
m(x,q,z):=\frac{1}{j(z;q)}\sum_{r=-\infty}^{\infty}\frac{(-1)^rq^{r(r-1)/2}z^r}{1-q^{r-1}xz},
$$
and 'building block' of mock theta functions
$$
f_{a,b,c}(x,y,q):=\sum_{sg(r)=sg(s)}sg(r)(-1)^{r+s}x^ry^sq^{ar(r-1)/2+brs+cs(s-1)/2},
$$
and
\begin{align*}
&g_{a,b,c}(x,y,q,z_1,z_0)\\
:=&\sum_{t=0}^{a-1}(-y)^tq^{ct(t-1)/2}j(q^{bt}x;q^a)m(-q^{ab(b+1)/2-ca(a+1)/2-t(b^2-ac)}\frac{(-y)^a}{(-x)^b},q^{a(b^2-ac)},z_0)\\
&+\sum_{t=0}^{c-1}(-x)^tq^{at(t-1)/2}j(q^{bt}y;q^c)m(-q^{cb(b+1)/2-ac(c+1)/2-t(b^2-ac)}\frac{(-x)^c}{(-y)^b},q^{a(b^2-ac)},z_1).
\end{align*}

We will use the formulas of Appell-Lerch series in \cite{Hi-Mo-14} such as the following theorems. Let $n=1$ in \cite[Theorem 1.6]{Hi-Mo-14}.
\begin{theorem}
\label{mthe16}
For generic $x,y\in \mathbb{C^*}$, we have
$$
f_{1,2,1}(x,y,q)=g_{1,2,1}(x,y,q,y/x,x/y).
$$
\end{theorem}

Let $n=1$ in \cite[Theorem 1.11]{Hi-Mo-14}.

\begin{theorem}
\label{mthe111}
For generic $x,y\in \mathbb{C^*}$, we have
$$
f_{1,5,1}(x,y,q)=g_{1,5,1}(x,y,q,y/x,x/y)-\Theta_{1,4}(x,y,q),
$$
where
$$
\Theta_{1,4}(x,y,q):=\frac{-qxyj(y/x;q^{24})}{j(y/x;q^{24})j(-q^{10}x^4;q^{24})j(-q^{10}y^4;q^{24})}\left(j(q^4;q^{16})S_1-qj(q^8;q^{16})S_2\right),
$$
with
\begin{align*}
S_1:=&\frac{j(q^{22}x^2y^2;q^{24})j(-q^{12}y/x;q^{24})j(q^5xy;q^{12})}{J_{12}^3J_{48}}\cdot \bigg(j(-q^{10}x^2y^2;q^{24})j(q^{12}y^2/x^2;q^{24})J_{24}^2\\
&+\frac{q^5x^2j(-q^{22}x^2y^2;q^{24})j(q^{12}y/x;q^{24})^2j(-y/x;q^{24})^2}{J_{24}}\bigg),
\end{align*}
\begin{align*}
S_2:=&\frac{j(q^{10}x^2y^2;q^{24})j(-y/x;q^{24})j(q^{11}xy;q^{12})}{J_{12}^2}\\
&\cdot \bigg(\frac{q^2j(-q^{10}x^2y^2;q^{24})j(q^{12}y^2/x^2;q^{24})J_{48}}{yJ_{24}}+\frac{qxj(-q^{22}x^2y^2;q^{24})j(q^{24}y^2/x^2;q^{48})^2}{J_{48}}\bigg).
\end{align*}
\end{theorem}

Some basic propositions of $m(x,q,z)$ can be obtained in \cite[Proposition 3.1]{Hi-Mo-14}.

\begin{prop}
\label{mpro31}
For generic $x,z \in \mathbb{C^*}$, we have
$$
m(x,q,z)=m(x,q,qz),
$$
$$
m(x,q,z)=x^{-1}m(x^{-1},q,z^{-1}),
$$
$$
m(qx,q,z)=1-xm(x,q,z),
$$
$$
m(x,q,z)=1-q^{-1}xm(q^{-1}x,q,z),
$$
$$
m(x,q,z)=x^{-1}-x^{-1}m(qx,q,z).
$$
\end{prop}

\begin{theorem}(\cite[Theorem 3.3]{Hi-Mo-14})
\label{mthe33}
For generic $x,z_0,z_1\in \mathbb{C^*}$, we have
$$
m(x,q,z_1)-m(x,q,z_0)=\frac{z_0J_1^3j(z_1/z_0;q)j(xz_0z_1;q)}{j(z_0;q)j(z_1;q)j(xz_0;q)j(xz_1;q)}.
$$
\end{theorem}

\begin{cor}(\cite[Corollary 3.7]{Hi-Mo-14})
\label{mcor37}
For generic $x,z\in \mathbb{C^*}$, we have
$$
m(x,q,z)=m(-qx^2,q^4,z^4)-\frac{x}{q}m(-x^2/q,q^4,q^4)-\frac{J_2J_4j(-xz^2;q)j(-xz^3;q)}{xj(xz;q)j(z^4;q^4)j(-qx^2z^4;q^2)}.
$$
\end{cor}

\subsection{Garvan's MAPLE package}

Define the usual Atkin $U_p$ operator which acts on a formal power series
$$
f(q)=\sum_{n\in \mathbb{Z}}a(n)q^n,
$$
by
$$
U_p(f(q))=\sum_{n\in \mathbb{Z}}a(pn)q^n.
$$
We will use Garvan's MAPLE programs to prove identities in this paper. The programs rely on the theory of modular functions. For details, see \cite[Section 2]{Ch-Ch-Ga}. Identities which have only eta-quotients (or with $U_p$ operator) can be proved by ETA-package algorithmically, see
\beq
\label{eta}
\text{https://qseries.org/fgarvan/qmaple/ETA/}
\eeq

\subsection{Derivative of theta functions}

We also need the derivative of certain theta functions.
\begin{lemma}We have
\label{pthe1}
\beq
\label{pj21}
\frac{d}{dz}j(\pm zq;q^2)\big|_{z=1}=0.
\eeq
\end{lemma}

\begin{proof}
By
$$
j(zq;q^2)=\sum_{n=-\infty}^{\infty}(-1)^nq^{n^2}z^n,
$$
we have
$$
\frac{d}{dz}j(zq;q^2)\big|_{z=1}=\sum_{n=-\infty}^{\infty}(-1)^nnq^{n^2}=0,
$$
and replace $q$ by $-q$
$$
\frac{d}{dz}j(-zq;q^2)\big|_{z=1}=\sum_{n=-\infty}^{\infty}nq^{n^2}=0.
$$
\end{proof}

Using the results by Lemke Oliver \cite[Theorem 1.1]{Ol-13}, we can prove more identities like \eqref{pj21}.

\begin{lemma}We have
\label{pthe2}
\begin{align}
\label{pj31}
\frac{d}{dz}\left(\frac{1}{z}j(z^6q;q^3)\right)\bigg|_{z=1}&=-\frac{J_1^4}{J_3}-9q\frac{J_9^3J_1}{J_3},\\
\label{pj31m}
\frac{d}{dz}\left(\frac{1}{z}j(-z^6q;q^3)\right)\bigg|_{z=1}&=-\frac{J_1^5}{J_2^2}.
\end{align}
\end{lemma}

\begin{proof}
Similar to Lemma \ref{pthe1}
$$
\frac{d}{dz}\left(\frac{1}{z}j(z^6q;q^3)\right)\bigg|_{z=1}=\sum_{n=-\infty}^{\infty}(-1)^n(6n-1)q^{n(3n-1)/2}.
$$
By the result of Lemke Oliver \cite[Theorem 1.1]{Ol-13}, we have
$$
J_1^3=\sum_{n=1}^{\infty}(-1)^{n-1}(2n-1)q^{n(n-1)/2}.
$$
Hence by 3-dissection of $J_1^3$
$$
U_3(J_1^3)=\sum_{n=-\infty}^{\infty}(-1)^{n-1}(6n-1)q^{n(3n-1)/2}.
$$
Then
$$
\frac{d}{dz}\left(\frac{1}{z}j(z^6q;q^3)\right)\bigg|_{z=1}=-U_3(J_1^3)=-\frac{J_1^4}{J_3}-9q\frac{J_9^3J_1}{J_3},
$$
where the last equation was verified by MAPLE \eqref{eta}. Also by \cite[Theorem 1.1]{Ol-13}
$$
\frac{d}{dz}\left(\frac{1}{z}j(-z^6q;q^3)\right)\bigg|_{z=1}=\sum_{n=-\infty}^{\infty}(6n-1)q^{n(3n-1)/2}=-\frac{J_1^5}{J_2^2}.
$$
\end{proof}

The following two lemmas can also be proved by \cite[Theorems 1.1 - 1.2]{Ol-13}. Since the proofs are similar to \ref{pthe2}, we omit the proofs.

\begin{lemma}We have
\begin{align}
\label{pj41}
\frac{d}{dz}\left(\frac{1}{z}j(z^4q;q^4)\right)\bigg|_{z=1}&=-\frac{J_2^9}{J_4^3J_1^3},\\
\label{pj41m}
\frac{d}{dz}\left(\frac{1}{z}j(-z^4q;q^4)\right)\bigg|_{z=1}&=-J_1^3.
\end{align}
\end{lemma}

\begin{lemma}We have
\begin{align}
\label{pj61}
\frac{d}{dz}\left(\frac{1}{z}j(z^3q;q^6)\right)\bigg|_{z=1}&=-\frac{J_2^5}{J_1^2},\\
\label{pj61m}
\frac{d}{dz}\left(\frac{1}{z}j(-z^3q;q^6)\right)\bigg|_{z=1}&=-\frac{J_4^2J_1^2}{J_2}.
\end{align}
\end{lemma}

The case for $q^{12}$ may be slightly more complicated.

\begin{lemma}We have
\begin{align}
\label{pj125m}
\frac{d}{dz}\left(\frac{1}{z}j(-z^{12}q^5;q^{12})\right)\bigg|_{z=1}&=-\frac{1}{2}\left(\frac{J_1^4}{J_3}+\frac{9qJ_9^3J_1}{J_3}+\frac{J_1^5}{J_2^2}\right),\\
\label{pj121m}
\frac{d}{dz}\left(\frac{1}{z^5}j(-z^{12}q;q^{12})\right)\bigg|_{z=1}&=-\frac{1}{2}\left(\frac{J_1^4}{qJ_3}+\frac{9J_9^3J_1}{J_3}-\frac{J_1^5}{qJ_2^2}\right).
\end{align}
\end{lemma}

\begin{proof}
Let
$$
f(z):=\frac{1}{z}j(-z^{12}q^5;q^{12}),
$$
and
$$
g(z):=\frac{1}{z^5}j(-z^{12}q;q^{12}).
$$
Then it is easy to check that
$$
f(z)+qg(1/z)=\frac{1}{z}j(-z^6q;q^3),
$$
and
$$
f(z)-qg(1/z)=\frac{1}{z}j(z^6q;q^3).
$$
 By \eqref{pj31} and \eqref{pj31m}, we have
$$
f'(1)-qg'(1)=-\frac{J_1^5}{J_2^2},
$$
and
$$
f'(1)+qg'(1)=-\frac{J_1^4}{J_3}-\frac{9qJ_9^3J_1}{J_3}.
$$
Hence the lemma holds.
\end{proof}

Replacing $q$ by $-q$ in \eqref{pj125m} and \eqref{pj121m}, we have

\begin{lemma}
\begin{align}
\label{pj125}
\frac{d}{dz}\left(\frac{1}{z}j(z^{12}q^5;q^{12})\right)\bigg|_{z=1}&=-\frac{1}{2}\left(\frac{J_{12}J_3J_2^{12}}{J_6^3J_4^4J_1^4}-\frac{9qJ_{18}^9J_{12}J_3J_2^3}{J_{36}^3J_9^3J_6^3J_4J_1}+\frac{J_2^{13}}{J_4^5J_1^5}\right),\\
\label{pj121}
\frac{d}{dz}\left(\frac{1}{z^5}j(z^{12}q;q^{12})\right)\bigg|_{z=1}&=\frac{1}{2}\left(\frac{J_{12}J_3J_2^{12}}{qJ_6^3J_4^4J_1^4}-\frac{9J_{18}^9J_{12}J_3J_2^3}{J_{36}^3J_9^3J_6^3J_4J_1}-\frac{J_2^{13}}{qJ_4^5J_1^5}\right).
\end{align}
\end{lemma}

Finally, it is easy to see that
\beq
\label{pj1}
\frac{d}{dz}\left(\frac{1}{z}j(-z^2;q)\right)\bigg|_{z=1}=0.
\eeq

\section{Hecke-Rogers series}

\subsection{Hecke-Rogers series of $\mathscr{F}_{4,-1}(q)$}

Following Alfes, Bringmann and Lovejoy \cite{Al-Br-Lo-11},  we define
$$
F_4(z,q):=-N^o(1,-1/q,-z,-q)=\sum_{n=0}^{\infty}\frac{(-1)^n(q;-q)_{2n}q^{n+1}}{(zq,q/z;q^2)_{n+1}}.
$$
It showed that $F_4(1,q)=\mathscr{F}_{4,-1}(q)$ and $F_4(i,-q)=-V_1(q)$ in \cite{Al-Br-Lo-11}. Mortenson \cite[Eq. (2.12)]{Mo-14} also studied this function and showed the Appell-Lerch series
\beq
\label{4mid}
\left(1-\frac{1}{z}\right)F_4(z,q)=m(-z,q^2,-q).
\eeq
The following theorem is equivalent to \eqref{4mid}.

\begin{theorem}
\label{4th1}
For $z\in \mathbb{C^*}$ and $z\neq 1$, we have
\begin{align}
\label{4th1id}
(zq,q/z,q^2;q^2)_\infty F_4(z,q)=&\frac{1}{1-z}\sum_{k=-\infty}^{\infty}\frac{(-1)^{k-1}q^{k^2}}{1+q^{2k-1}}z^{1-k}=-\frac{1}{1-z}\sum_{k=-\infty}^{\infty}\frac{(-1)^{k-1}q^{k^2}}{1+q^{2k-1}}z^k\\
\nonumber
=&\sum_{k=1}^{\infty}\frac{(-1)^{k-1}q^{k^2}}{1+q^{2k-1}}\cdot \frac{z^{1-k}-z^k}{1-z}
\end{align}
\end{theorem}

\begin{proof}
By Theorem \ref{mthe33}
$$
m(-z,q^2,-q)=m(-z,q^2,q/z).
$$
In addition by Proposition \ref{mpro31}
$$
m(-z,q^2,q/z)=-\frac{1}{z}m(-1/z,q^2,z/q)=-\frac{1}{z}m(-1/z,q^2,qz).
$$
Then \eqref{4th1id} holds obviously.
\end{proof}

Letting $z\rightarrow 1$ in \eqref{4th1id}, we have the Appell-Lerch series of $\mathscr{F}_{4,-1}(q)$.
\begin{cor}
\beq
\label{4cor1}
\frac{J_1^2}{J_2}\mathscr{F}_{4,-1}(q)=\sum_{k=1}^{\infty}\frac{(-1)^{k-1}(2k-1)q^{k^2}}{1+q^{2k-1}}.
\eeq
\end{cor}

Then we prove the Hecke-Rogers series of $F_4(z,q)$.

\begin{theorem}
\label{4th2}
For $z\in \mathbb{C^*}$ and $z\neq 1$, we have
\beq
\label{4th2id}
(-zq,-1/z,q;q)_\infty F_4(z,-q)=\sum_{1\leq j\leq n}q^{n^2-j(j-1)/2}\cdot \frac{z^n-z^{-n}}{1-z}.
\eeq
\end{theorem}

\begin{proof}
Replacing $q$ by $-q$ in \eqref{4th1id}, we have
$$
(-zq,-q/z,q^2;q^2)_\infty F_4(z,-q)=\frac{1}{1-z}\sum_{k=-\infty}^{\infty}\frac{q^{k^2}}{1-q^{2k-1}}z^k.
$$
Hence \eqref{4th2id} is equivalent to
\beq
\label{4th21}
(-zq,-1/z,q;q)_\infty \sum_{k=-\infty}^{\infty}\frac{q^{k^2}}{1-q^{2k-1}}z^k=(-zq,-q/z,q^2;q^2)_\infty \sum_{1\leq j\leq |n|}sg(n)q^{n^2-j(j-1)/2}z^n.
\eeq
Noting that
$$
\frac{(-zq,-1/z,q;q)_\infty}{(-zq,-q/z,q^2;q^2)_\infty}=\frac{J_1}{J_2^2}(-zq^2,-1/z;q^2)_\infty,
$$
\eqref{4th21} can be simplified as
$$
\sum_{m=-\infty}^{\infty}q^{m(m+1)}z^m\sum_{k=-\infty}^{\infty}\frac{q^{k^2}}{1-q^{2k-1}}z^k=\frac{J_2^2}{J_1}\sum_{1\leq j\leq |n|}sg(n)q^{n^2-j(j-1)/2}z^n.
$$
Denote $A_n$ and $B_n$ to be the coefficients of $z^n$ on both sides, respectively. Then
$$
A_n=\sum_{k=-\infty}^{\infty}\frac{q^{k^2+(n-k)(n-k+1)}}{1-q^{2k-1}},
$$
and
$$
B_n=\frac{J_2^2}{J_1}\sum_{j=1}^{|n|}sg(n)q^{n^2-j(j-1)/2}.
$$
It is easy to verify that $B_n$ satisfy
\begin{enumerate}
\item[(1)]For all $n\in \mathbb{Z}$
$$
B_n+B_{-n}=0,
$$
\item[(2)]For all $n\in \mathbb{N}$
$$
\frac{B_{n+1}}{q^{(n+1)^2}}-\frac{B_n}{q^{n^2}}=q^{-\frac{n(n+1)}{2}}\frac{J_2^2}{J_1}.
$$
\end{enumerate}
Noting that (1) implies $B_0=0$. If $A_n$ also satisfy (1) and (2), then $A_n=B_n$ for all $n\in \mathbb{Z}$ by mathematical induction. Since
$$
A_{-n}=\sum_{k=-\infty}^{\infty}\frac{q^{k^2+(-n-k)(-n-k+1)}}{1-q^{2k-1}}=-\sum_{k=-\infty}^{\infty}\frac{q^{k^2+(n-k)(n-k+1)}}{1-q^{2k-1}}=-A_n,
$$
by replacing $k$ by $1-k$,  it shows that $A_n$ satisfy (1). On the other hand,
$$
A_{n}=\sum_{k=-\infty}^{\infty}\frac{q^{k^2+(n-k)(n-k+1)}}{1-q^{2k-1}}=q^{n^2}\sum_{k=-\infty}^{\infty}\frac{q^{(k-n)(2k-1)}}{1-q^{2k-1}}.
$$
We find
\beq
\label{4th22}
\frac{A_{n+1}}{q^{(n+1)^2}}-\frac{A_n}{q^{n^2}}=\sum_{k=-\infty}^{\infty}\frac{q^{(k-n-1)(2k-1)}-q^{(k-n)(2k-1)}}{1-q^{2k-1}}=\sum_{k=-\infty}^{\infty}q^{(k-n-1)(2k-1)}.
\eeq
Let
\beq
\label{4th24}
C_n:=\sum_{k=-\infty}^{\infty}q^{(2k-n)(2k+1-n)/2}=\sum_{k=-\infty}^{\infty}q^{(k-n-1)(2k-1)+n(n+1)/2}.
\eeq
Then
$$
C_{n+1}=\sum_{k=-\infty}^{\infty}q^{(2k-n-1)(2k-n)/2}=\sum_{k=-\infty}^{\infty}q^{(2(n-k)-n+1)(2(n-k)-n)/2}=C_n.
$$
That is
\beq
\label{4th25}
C_n=C_0=\frac{J_2^2}{J_1}.
\eeq
By \eqref{4th22} - \eqref{4th25}
$$
\frac{A_{n+1}}{q^{(n+1)^2}}-\frac{A_n}{q^{n^2}}=q^{-\frac{n(n+1)}{2}}C_n=q^{-\frac{n(n+1)}{2}}\frac{J_2^2}{J_1},
$$
we find that $A_n$ also satisfy (2). Hence $A_n=B_n$ for all $n\in \mathbb{Z}$ and \eqref{4th21} holds.
\end{proof}

From \eqref{4th2id}, we can obtain the Hecke-Rogers series of $V_1(q)$ by $z=i$, which is equivalent to \cite[Eq. (2.37)]{Cu-Gu-Ha-18}
\beq
\label{4V1HR}
\frac{J_1J_4}{J_2}V_1(q)=\sum_{1\leq j\leq n}\left(\frac{-4}{n}\right)q^{n^2-j(j-1)/2},
\eeq
where $(\frac{\cdot}{\cdot})$ denote the Kronecker symbol. Replacing $q$ by $-q$ and letting $z\rightarrow 1$ in \eqref{4th2id}, we have the Hecke-Rogers series of $\mathscr{F}_{4,-1}(q)$ which is similar  to \eqref{4V1HR}.
\begin{cor}We have
$$
\frac{J_1J_4}{J_2}\mathscr{F}_{4,-1}(q)=\sum_{1\leq j\leq n}(-1)^{n-1+j(j-1)/2}nq^{n^2-j(j-1)/2}.
$$
\end{cor}

\subsection{Hecke-Rogers series of $\mathscr{F}_{8,-1}(q)$}

Let
$$
F_8(z,q):=\sum_{n=0}^{\infty}\frac{(-1)^n(q;q^2)_nq^{(n+1)^2}}{(zq,q/z;q^2)_{n+1}}.
$$
This function was also studied by Alfes, Bringmann and Lovejoy \cite{Al-Br-Lo-11} where showed that $F_8(1,q)=N^o(0,-1,1,q)=\mathscr{F}_{8,-1}(q)$ and $F_8(-1,-q)=-N^o(0,1,1,q)=-A(q)$ in page 3, and Mortenson \cite{Mo-14} where showed  the Appell-Lerch series of $F_8(z,q)$. In this subsection, again, we use the formulas of Appell-Lerch series to rewrite $F_8(z,q)$ as another form and then prove the Hecke-Rogers identity.

\begin{lemma}
\label{8lem1}
If $z\in \mathbb{C^*}$ is not an integral power of $q$, then
\beq
\label{8mlem}
m(-z,q,-1)=m(-qz^2,q^4,q^2/z^2)-\frac{1}{z}m(-q/z^2,q^4,q^2z^2)+\frac{j(q;q^2)^2}{2j(z;q)}.
\eeq
\end{lemma}

\begin{proof}
Setting $x=-z$, $z_1=-1$ and $z_0=\sqrt{z/q}$ in Theorem \ref{mthe33}, we have
\beq
\label{8m0}
m(-z,q,-1)=m(-z,q,\sqrt{z/q})+\frac{j(q;q^2)^2}{2j(z;q)}.
\eeq
Then by Corollary \ref{mcor37} and Proposition \ref{mpro31}
\begin{align}
\label{8m1}
m(-z,q,\sqrt{z/q})=&m(-qz^2,q^4,q^2/z^2)+\frac{z}{q}m(-z^2/q,q^4,q^2/z^2)\\
\nonumber
=&m(-qz^2,q^4,q^2/z^2)-\frac{1}{z}m(-q/z^2,q^4,z^2/q^2)\\
\nonumber
=&m(-qz^2,q^4,q^2/z^2)-\frac{1}{z}m(-q/z^2,q^4,q^2z^2).
\end{align}
\eqref{8mlem} holds by \eqref{8m0} and \eqref{8m1}.
\end{proof}

\begin{theorem}
\label{8th1}
For $z\in \mathbb{C^*}$ and $z\neq 1$, we have
\beq
\label{8th1id}
(z^2q^2,q^2/z^2,q^4;q^4)_\infty F_8(z,q)=\sum_{k=-\infty}^{\infty}\frac{(-1)^{k-1}q^{2k^2}}{1+q^{4k-1}}\cdot \frac{z^{1-2k}-z^{2k}}{1-z}.
\eeq
\end{theorem}

\begin{proof}
If $z$ is an integral power of $q$, it is easy to check that both sides of \eqref{8th1} are zero. For $z$ is not an integral power of $q$, setting $x=-z$ in \cite[Eq. (2.15)]{Mo-14}
$$
F_8(z,q)=\frac{-z}{1-z}\left(m(-z,q,-1)-\frac{j(q;q^2)^2}{2j(z;q)}\right).
$$
By Lemma \ref{8lem1}
\begin{align*}
F_8(z,q)=&\frac{1}{1-z}(m(-q/z^2,q^4,q^2z^2)-zm(-qz^2,q^4,q^2/z^2))\\
=&\frac{1}{j(q^2z^2;q^4)}\sum_{k=-\infty}^{\infty}\frac{(-1)^{k-1}q^{2k^2}}{1+q^{4k-1}}\cdot \frac{z^{1-2k}-z^{2k}}{1-z},
\end{align*}
which is \eqref{8th1id}.
\end{proof}

Letting $z\rightarrow 1$, \eqref{8th1id} yields
\begin{cor}
\beq
\label{8ALid}
\frac{J_2^2}{J_4}\mathscr{F}_{8,-1}(q)=\sum_{k=-\infty}^{\infty}\frac{(-1)^{k-1}(4k-1)q^{2k^2}}{1+q^{4k-1}}.
\eeq
\end{cor}
By replacing $q$ by $-q$ and $z=-1$,  \eqref{8th1id} yields \cite[Eq. (5.1)]{Hi-Mo-14}
$$
\frac{J_2^2}{J_4}A(q)=\sum_{k=-\infty}^{\infty}\frac{(-1)^{k-1}q^{2k^2}}{1-q^{4k-1}}.
$$
 These easily imply that \cite[Lemma 3.1]{Ch-Ga}
$$
\mathscr{F}_{8,-1}(q)\equiv -A(-q) \pmod 4.
$$

\begin{theorem}
\label{8th2}
For $z\in \mathbb{C^*}$ and $z\neq 1$, we have
\beq
\label{8th2id}
(zq,q/z,q^2;q^2)_\infty F_8(z,q)=\sum_{1\leq j\leq |n|}sg(n)(-1)^{j-1}q^{2n^2-n-j^2+j}\cdot \frac{z^{1-j}-z^j}{1-z}.
\eeq
\end{theorem}

\begin{proof}
By Theorem \ref{8th1}, for $z\in \mathbb{C^*}$, \eqref{8th2id} is equivalent to
\begin{align*}
(1-z)F_8(z,q)=&\frac{1}{(z^2q^2,q^2/z^2,q^4;q^4)_\infty} \sum_{k=-\infty}^{\infty}\frac{(-1)^{k-1}q^{2k^2}}{1+q^{4k-1}}(z^{1-2k}-z^{2k})\\
=&\frac{1}{(zq,q/z,q^2;q^2)_\infty} \sum_{1\leq j\leq |n|}sg(n)(-1)^{j-1}q^{2n^2-n-j^2+j}(z^{1-j}-z^j),
\end{align*}
which can be simplified as
\begin{align}
\label{8th21}
&\sum_{m=-\infty}^{\infty}(-1)^mq^{2m^2} \sum_{k=-\infty}^{\infty}\frac{(-1)^{k-1}q^{2k^2}}{1+q^{4k-1}}(z^{1-2k}-z^{2k})\\
\nonumber
=&\sum_{m=-\infty}^{\infty}q^{m^2}z^m \sum_{1\leq j\leq |n|}sg(n)(-1)^{j-1}q^{2n^2-n-j^2+j}(z^{1-j}-z^j).
\end{align}
Denote $A_k$ and $B_k$ to be the coefficients of $z^k$ on both sides of \eqref{8th21}, respectively. It is easy to see that
$$
A_{2k}=-A_{1-2k}=\frac{J_2^2}{J_4}\cdot \frac{(-1)^kq^{2k^2}}{1+q^{4k-1}}.
$$
For $B_k$, by \eqref{8th21} we have
$$
B_k=\sum_{1\leq j\leq |n|}sg(n)(-1)^{j-1}q^{2n^2-n-j^2+j}\left(q^{(k-1+j)^2}-q^{(k-j)^2}\right).
$$
So we can verify that
$$
B_{2k}=-B_{1-2k},
$$
and
\begin{align*}
B_{2k}=&\sum_{1\leq j\leq |n|}sg(n)(-1)^{j-1}q^{2n^2-n-j^2+j}\left(q^{(2k-1+j)^2}-q^{(2k-j)^2}\right)\\
=&\sum_{n=-\infty}^{\infty}sg(n)q^{4k^2+2n^2-n}\left(\frac{1-(-1)^nq^{|n|(4k-1)}}{1+q^{4k-1}}- \frac{1-(-1)^nq^{-|n|(4k-1)}}{1+q^{4k-1}}\right)\\
=&\frac{q^{4k^2}}{1+q^{4k-1}}\sum_{n=-\infty}^{\infty}sg(n)(-1)^nq^{2n^2-n}(q^{-|n|(4k-1)}-q^{|n|(4k-1)})\\
=&\frac{q^{4k^2}}{1+q^{4k-1}}\left(\sum_{n=-\infty}^{\infty}(-1)^nq^{2n^2-4nk}-\sum_{n=-\infty}^{\infty}(-1)^nq^{2n^2-2n+4nk}\right)\\
=&\frac{(-1)^kq^{2k^2}}{1+q^{4k-1}}\sum_{n=-\infty}^{\infty}(-1)^{n-k}q^{2(n-k)^2}\\
=&\frac{J_2^2}{J_4}\cdot \frac{(-1)^kq^{2k^2}}{1+q^{4k-1}}.
\end{align*}
Hence $A_k=B_k$ for all integers $k$ and \eqref{8th21} holds.
\end{proof}

By replacing $q$ by $-q$ and $z=-1$, \eqref{8th2id} yields the Hecke-Rogers series of $A(q)$ \cite[Eq. (2.24)]{Cu-Gu-Ha-18}
\beq
\label{8AHR}
\frac{J_1^2}{J_2}A(q)=\sum_{1\leq j\leq |n|}sg(n)(-1)^{n-1}q^{2n^2-n-j^2+j}.
\eeq
Letting $z\rightarrow 1$ in \eqref{8th2id}, we have the Hecke-Rogers series of $\mathscr{F}_{8,-1}(q)$ \cite[P. 376]{Hu-07} which is similar to \eqref{8AHR}.
\begin{cor}We have
\beq
\label{8hrid}
\frac{J_1^2}{J_2}\mathscr{F}_{8,-1}(q)=\sum_{1\leq j\leq |n|}sg(n)(-1)^{j-1}(2j-1)q^{2n^2-n-j^2+j}.
\eeq
\end{cor}

\subsection{Hecke-Rogers series of $\mathscr{F}_{12,-1}(q)$}

Unlike $\mathscr{F}_{4,-1}(q)$ and $\mathscr{F}_{8,-1}(q)$, we did not find Eulerian type series of $\mathscr{F}_{12,-1}(q)$ and $\mathscr{F}_{24,-1}(q)$. So we did not find a nice $z$-analog of $\mathscr{F}_{12,-1}(q)$ and mock theta functions $\sigma(q)$. By the 3-dissection of $\mathscr{F}_{4,-1}(q)$, we have the Appell-Lerch series which is similar  to \cite[Eq. (4.8)]{An-Hi-91} with $z=q^3$
$$
\frac{J_3^2}{J_6}\sigma(q)=\sum_{k=-\infty}^{\infty}\frac{(-1)^{k-1}q^{3k^2}}{1-q^{6k-1}}.
$$

\begin{theorem}We have
$$
\frac{J_3^2}{J_6}\mathscr{F}_{12,-1}(q)=\sum_{k=-\infty}^{\infty}\frac{(-1)^{k-1}(6k-1)q^{3k^2}}{1+q^{6k-1}}+\frac{2qJ_{12}J_6J_2^5}{J_4J_1^2}.
$$
\end{theorem}

\begin{proof}
We denote the 3-dissection on both sides of \eqref{4cor1} by
$$
\frac{J_1^2}{J_2}=A_0(q^3)+qA_1(q^3)+q^2A_2(q^3),
$$
$$
\mathscr{F}_{4,-1}(q)=B_0(q^3)+qB_1(q^3)+q^2B_2(q^3),
$$
and
$$
\sum_{k=1}^{\infty}\frac{(-1)^{k-1}(2k-1)q^{k^2}}{1+q^{2k-1}}=C_0(q^3)+qC_1(q^3)+q^2C_2(q^3).
$$
Then
\beq
\label{12lid2}
A_0(q)B_0(q)+qA_1(q)B_2(q)+qA_2(q)B_1(q)=C_0(q).
\eeq
We have
$$
B_0(q)=\mathscr{F}_{12,-1}(q),
$$
and we note that
$$
\frac{J_1^2}{J_2}=\sum_{k=-\infty}^{\infty}(-1)^kq^{k^2}=\sum_{k=-\infty}^{\infty}(-1)^kq^{9k^2}-2q\sum_{k=-\infty}^{\infty}(-1)^kq^{9k^2+6k}=\frac{J_9^2}{J_{18}}-2q\frac{J_{18}^2J_3}{J_9J_6}.
$$
Hence
\beq
\label{12lemid24}
A_0(q)=\frac{J_3^2}{J_6},\text{  }A_1(q)=-2\frac{J_{6}^2J_1}{J_3J_2}\text{  and  }A_2(q)=0.
\eeq
By \cite[Eq. (2.7), (2.9)]{Ch-Ga}
\begin{align}
\label{12lmb}
B_2(q)=&\mathscr{F}_{12,7}(q)=\sum_{n=0}^{\infty}H(24n+7)q^{2n}+3q\sum_{n=0}^{\infty}H(24n+19)q^{2n}\\
\nonumber
=&\frac{J_6^2J_4^5}{J_{12}J_2^3}+3q\frac{J_{12}^3J_4}{J_2}=\frac{J_{12}J_3J_2^6}{J_6J_4J_1^3},
\end{align}
where the last equation was verified by MAPLE \eqref{eta}. Since
$$
\sum_{k=1}^{\infty}\frac{(-1)^{k-1}(2k-1)q^{k^2}}{1+q^{2k-1}}=\sum_{k=1}^{\infty}\frac{(-1)^{k-1}(2k-1)q^{k^2}}{1+q^{6k-3}}(1-q^{2k-1}+q^{4k-2}),
$$
we have
\begin{align}
\label{12lmc}
&C_0(q^3)\\
\nonumber
=&\sum_{k=1}^{\infty}\frac{(-1)^{3k-1}(2\cdot(3k)-1)q^{(3k)^2}}{1+q^{6\cdot(3k)-3}}+\sum_{k=0}^{\infty}\frac{(-1)^{3k+1-1}(2\cdot(3k+1)-1)q^{(3k+1)^2+4\cdot(3k+1)-2}}{1+q^{6\cdot(3k+1)-3}}\\
\nonumber
=&\sum_{k=-\infty}^{\infty}\frac{(-1)^{k-1}(6k-1)q^{9k^2}}{1+q^{18k-3}}.
\end{align}
Setting $q\rightarrow q^{1/3}$, it becomes
$$
C_0(q)=\sum_{k=-\infty}^{\infty}\frac{(-1)^{k-1}(6k-1)q^{3k^2}}{1+q^{6k-1}}.
$$
Hence \eqref{12lid2} becomes
$$
\frac{J_3^2}{J_6}\mathscr{F}_{12,-1}(q)-\frac{2qJ_{12}J_6J_2^5}{J_4J_1^2}=\sum_{k=-\infty}^{\infty}\frac{(-1)^{k-1}(6k-1)q^{3k^2}}{1+q^{6k-1}}.
$$
\end{proof}

The following lemmas are needed to prove the Hecke-Rogers identity of $\mathscr{F}_{12,-1}(q)$.

\begin{lemma}
\label{12lem1}
Let
$$
g(z):=\frac{1}{z}j(-z^2;q^2)m(q,q^6,-z^2/q).
$$
Then
$$
g'(1)=-\frac{J_6^{12}J_4^4J_1^3}{J_{12}^6J_3^3J_2^6}.
$$
\end{lemma}

\begin{proof}
From Theorem \ref{mthe33}
$$
m(q,q^6,-z^2/q)=m(q,q^6,-1)-\frac{J_6^3j(z^2/q;q^6)j(z^2;q^6)}{j(-1;q^6)j(-q;q^6)j(-z^2/q;q^6)j(-z^2;q^6)}.
$$
Let
$$
f(z):=\frac{j(z^2/q;q^6)j(z^2;q^6)}{j(-z^2/q;q^6)j(-z^2;q^6)}.
$$
Then
\begin{align}
\label{12lid1}
f'(1)=&\lim_{z\rightarrow 1}\frac{f(z)-f(1)}{z-1}\\
\nonumber
=&\lim_{z\rightarrow 1}(1+z)(q^6z,q^6/z,q^6;q^6)_\infty\frac{j(z^2/q;q^6)}{j(-z^2/q;q^6)j(-z^2;q^6)}\\
\nonumber
=&\frac{J_6^7J_4J_1^2}{J_{12}^3J_3^2J_2^3}.
\end{align}
We find that
$$
\frac{d}{dz}m(q,q^6,-z^2/q)\bigg|_{z=1}=-\frac{J_6^{12}J_4^2J_1^3}{2J_{12}^6J_3^3J_2^5},
$$
and by \eqref{pj1}
$$
\frac{d}{dz}\left(\frac{1}{z}j(-z^2;q^2)\right)\bigg|_{z=1}=0.
$$
Hence
$$
g'(1)=-j(-1;q^2)\frac{J_6^{12}J_4^2J_1^3}{2J_{12}^6J_3^3J_2^5}=-\frac{J_6^{12}J_4^4J_1^3}{J_{12}^6J_3^3J_2^6}.
$$
\end{proof}

\begin{lemma}
Let
$$
g(z):=\frac{1}{z}j(qz^4;q^2)m(-q^2/z^6,q^6,q^3z^6).
$$
Then
$$
g'(1)=-\frac{J_6J_1^2}{J_3^2J_2}\sum_{k=-\infty}^{\infty}\frac{(-1)^{k-1}(6k-1)q^{3k^2}}{1+q^{6k-1}}.
$$
\end{lemma}

\begin{proof}
Since by \eqref{pj21}
$$
\frac{d}{dz}j(qz^4;q^2)\bigg|_{z=1}=0,
$$
we have
\begin{align*}
g'(1)=&j(q;q^2)\frac{d}{dz}\left(\frac{1}{z}m(-q^2/z^6,q^6,q^3z^6)\right)\bigg|_{z=1}\\
=&j(q;q^2)\left(\frac{d}{dz}m(-q^2/z^6,q^6,q^3z^6)\bigg|_{z=1}-m(-q^2,q^6,q^3)\right)\\
=&-\frac{J_6J_1^2}{J_3^2J_2}\sum_{k=-\infty}^{\infty}\frac{(-1)^{k-1}(6k-1)q^{3k^2}}{1+q^{6k-1}}.
\end{align*}
\end{proof}

\begin{lemma}
\label{12lem3}
Let
$$
g(z):=\frac{j(qz^4;q^2)j(-q^4z^4;q^6)j(q^4z^2;q^6)}{zj(-q^5z^2;q^6)j(q^3z^6,q^6)j(-q^5;q^6)j(q^5z^4;q^6)}.
$$
Then
$$
g'(1)=-\frac{J_6^9J_4^4J_1^3}{J_{12}^6J_3^3J_2^6}-\frac{2qJ_{12}J_2^4}{J_6J_4J_3^2}.
$$
\end{lemma}

\begin{proof}
Let
$$
g_1(z):=\frac{j(q;q^2)}{j(q^3;q^6)j(-q^5,q^6)}\cdot \frac{j(-z^4q^4;q^6)j(z^2q^4;q^6)}{zj(-z^2q^5;q^6)j(z^4q^5;q^6)}.
$$
Then by \eqref{pj21}
$$
g_1'(1)=g'(1).
$$
Let
\begin{align*}
u_1(z):=&z^{2/3}j(-z^4q^4;q^6)=z^{2/3}j(-q^2/z^4;q^6),\\
u_2(z):=&z^{1/3}j(z^2q^4;q^6)=z^{1/3}j(q^2/z^2;q^6),\\
v_1(z):=&z^{2/3}j(-z^2q^5;q^6)=z^{2/3}j(-q/z^2;q^6),\\
v_2(z):=&z^{4/3}j(z^4q^5;q^6)=z^{4/3}j(q/z^4;q^6).
\end{align*}
By \eqref{pj31}, \eqref{pj31m}, \eqref{pj61} and \eqref{pj61m}, we have
\begin{align*}
g_1'(1)=&g_1(1)\left(\frac{u_1'(1)}{u_1(1)}+\frac{u_2'(1)}{u_2(1)}-\frac{v_1'(1)}{v_1(1)}-\frac{v_2'(1)}{v_2(1)}\right)\\
=&\frac{2J_6J_2^2J_1^3}{3J_{12}^2J_3^3}+\frac{J_6^3J_4^3J_1^3}{J_{12}^3J_3^3J_2^4}\left(\frac{J_2^{3}}{3J_6}+\frac{3q^2J_{18}^3}{J_6}\right)-\frac{2J_6^4J_4^6J_1^6}{3J_{12}^4J_3^4J_2^7}-\frac{4J_6J_4^3J_2^2}{3J_{12}^3J_3^2}.\\
=&-\frac{J_6^9J_4^4J_1^3}{J_{12}^6J_3^3J_2^6}-\frac{2qJ_{12}J_2^4}{J_6J_4J_3^2},
\end{align*}
where the last equation was verified by MAPLE \eqref{eta}.
\end{proof}

Now we can prove the Hecke-Rogers identity.

\begin{theorem}We have
\beq
\label{12th2id}
\frac{J_1^2}{J_2}\mathscr{F}_{12,-1}(q)=\sum_{1-|n|\leq j\leq |n|}sg(n)(-1)^{j-1}(4n-1)q^{4n^2-2n-3j^2+2j}.
\eeq
\end{theorem}

\begin{proof}
Let
$$
f(z):=qz^3f_{1,2,1}(q^3z^4,-q^4z^2,q^2).
$$
The right-hand side of \eqref{12th2id} is
\begin{align*}
&\sum_{1-|n|\leq j\leq |n|}sg(n)(-1)^{j-1}(4n-1)q^{4n^2-2n-3j^2+2j}\\
=&\sum_{sg(r)=sg(s)}(-1)^r(2s+4r+3)q^{(2r+1)s+(r+s+1)^2}\\
=&f'(1).
\end{align*}
Hence \eqref{12th2id} is equivalent to
\beq
\label{12thid1}
\frac{J_2}{J_1^2}f'(1)=\frac{J_6}{J_3^2}\sum_{k=-\infty}^{\infty}\frac{(-1)^{k-1}(6k-1)q^{3k^2}}{1+q^{6k-1}}+\frac{2qJ_{12}J_6^2J_2^5}{J_4J_3^2J_1^2}.
\eeq
By Theorem \ref{mthe16}
\beq
\label{12thid2}
f(z)=\frac{1}{z}(j(-z^2;q^2)m(q,q^6,-z^2/q)-j(qz^4;q^2)m(-q^2/z^6,q^6,-q^5z^2)).
\eeq
In addition, by Theorem \ref{mthe33}
\beq
\label{12thid3}
m(-q^2/z^6,q^6,-q^5z^2)=m(-q^2/z^6,q^6,q^3z^6)+\frac{J_6^3j(-q^4z^4;q^6)j(q^4z^2;q^6)}{j(-q^5z^2;q^6)j(q^3z^6;q^6)j(-q^5q^6)j(q^5z^4;q^6)}.
\eeq
\eqref{12thid1} holds by \eqref{12thid2} - \eqref{12thid3} and Lemma \ref{12lem1} - Lemma \ref{12lem3}.
\end{proof}

We rewrite \cite[Eq. (4.2)]{An-Hi-91} as
\beq
\label{12si}
\frac{J_1^2}{J_2}\sigma(q)=\sum_{1-|n|\leq j\leq |n|}sg(n)(-1)^{j-1}q^{4n^2-2n-3j^2+2j}.
\eeq
\eqref{12th2id} and \eqref{12si}  imply that
$$
\mathscr{F}_{12,-1}(q)\equiv -\sigma(q) \pmod 4,
$$
which is equivalent to \cite[Lamma 3.11]{Ch-Ga}.

\subsection{Hecke-Rogers series of $\mathscr{F}_{24,-1}(q)$}

We first prove the Appell-Lerch series of $\mathscr{F}_{24,-1}(q)$ which is more complex than $\mathscr{F}_{12,-1}(q)$.
\begin{theorem}
\label{24th1}
$$
\frac{J_6^2}{J_{12}}\mathscr{F}_{24,-1}(q)=\sum_{k=-\infty}^{\infty}\frac{(-1)^{k-1}(12k-1)q^{6k^2}}{1+q^{12k-1}}+\sum_{k=-\infty}^{\infty}\frac{(-1)^{k-1}(12k-7)q^{6k^2-2}}{1+q^{12k-7}}+\frac{2qJ_{12}^2J_3^2J_2^6}{J_6^2J_4J_1^3}.
$$
\end{theorem}

\begin{proof}
We denote the 3-dissection on both sides of \eqref{8ALid} by
$$
\frac{J_2^2}{J_4}=A_0(q^3)+qA_1(q^3)+q^2A_2(q^3),
$$
$$
\mathscr{F}_{8,-1}(q)=B_0(q^3)+qB_1(q^3)+q^2B_2(q^3),
$$
and
$$
\sum_{k=-\infty}^{\infty}\frac{(-1)^{k-1}(4k-1)q^{2k^2}}{1+q^{4k-1}}=C_0(q^3)+qC_1(q^3)+q^2C_2(q^3).
$$
We obtain that
\beq
\label{24th1id1}
A_0(q)B_0(q)+qA_2(q)B_1(q)+qA_1(q)B_2(q)=C_0(q).
\eeq
Similar to \eqref{12lemid24}-\eqref{12lmc} we have
\begin{align}
\label{24lma}
A_0(q)&=\frac{J_6^2}{J_{12}},\\
A_1(q)&=0,\\
A_2(q)&=\frac{-2J_{12}^2J_2}{J_6J_4},\\
B_0(q)&=\mathscr{F}_{24,-1}(q),\\ B_1(q)&=\mathscr{F}_{24,7}(q)=\sum_{n=1}^{\infty}H(24n+7)q^n=\frac{J_3^2J_2^5}{J_6J_1^3},\\
\label{24lmc}
C_0(q)&=\sum_{k=-\infty}^{\infty}\frac{(-1)^{k-1}(12k-1)q^{6k^2}}{1+q^{12k-1}}+\sum_{k=-\infty}^{\infty}\frac{(-1)^{k-1}(12k-7)q^{6k^2-2}}{1+q^{12k-7}}.
\end{align}
We complete the proof by substituting \eqref{24lma} - \eqref{24lmc} into \eqref{24th1id1}.
\end{proof}

The following lemma can be proved similar to Lemma \ref{12lem3}. We use \eqref{pj21} - \eqref{pj1} to calculate the sum of  each term with the  eta-quotients, and  then use MAPLE \eqref{eta} to verify the eta-quotients identity. We omit the proofs.

\begin{lemma}
\label{24lmt}
Let
\begin{align*}
g(z):=&\frac{z^5J_{12}J_4^2j(-qz^4;q^{12})}{qJ_8J_6j(-z^8;q^{12})j(-q^4z^8;q^{12})}\left(\frac{J_{12}^2J_8}{J_{24}^2J_4}j(q^{14}z^8;q^{24})^2-\frac{q^2J_{24}^2J_6J_4^2}{J_{12}^2J_8J_2}j(q^8z^8;q^{12})\right)\\
-&\frac{zJ_{12}^3j(qz^4;q^2)j(q^5z^8;q^{12})}{j(qz^4;q^{12})j(q^6z^{12};q^{12})}\left(\frac{j(-q^2z^4;q^{12})}{j(-q^4z^8;q^{12})j(-q^{11};q^{12})}+\frac{z^4j(-q^6z^4;q^{12})}{qj(-z^8;q^{12})j(-q^5;q^{12})}\right).
\end{align*}
Then
$$
g'(1)=-\frac{2qJ_{12}^3J_3^2J_2^5}{J_6^4J_4J_1}.
$$
\end{lemma}

\begin{theorem} The Hecke-Rogers series of $\mathscr{F}_{24,-1}(q)$ is
\beq
\label{24mid}
\frac{J_1^2}{J_2}\mathscr{F}_{24,-1}(q)=\sum_{1-|n|\leq j\leq |n|}sg(n)(-1)^{n-1}(4j-1)q^{3n^2-n-2j^2+j}.
\eeq
\end{theorem}

\begin{proof}
Let
$$
f(z):=\frac{q_1}{z}f_{1,5,1}(iz^2q_1^7,iq_1^3/z^2,q_1^2).
$$
Letting $q=q_1^4$, the right-hand side of \eqref{24mid} is
\begin{align*}
&\sum_{1-|n|\leq j\leq |n|}sg(n)(-1)^{n-1}(4j-1)q_1^{12n^2-4n-8j^2+4j}\\
=&\sum_{\substack{sg(r)=sg(s) \\r+s\equiv 1\pmod 2}}sg(r)(-1)^r(2r-2s-1)q_1^{r^2+6r+10rs+s^2+2s+1}\\
=&\IM\left(\sum_{sg(r)=sg(s)}sg(s)(-i)^{r+s}(2r-2s-1)q_1^{r^2+6r+10rs+s^2+2s+1}\right)\\
=&\IM(f'(1)),
\end{align*}
where $\IM(z)$ denotes the imaginary part of $z$. Then by Theorem \ref{mthe111} we have
$$
\IM(f(z))=\frac{q_1}{z}\IM\left(g_{1,5,1}(iq_1^7z^2,iq_1^3/z^2,q_1^2,1/q_1^4z^4,q_1^4z^4)-\Theta_{1,4}(iq_1^7z^2,iq_1^3/z^2,q_1^2)\right),
$$
where it is easy to calculate that
\begin{align*}
&\frac{q_1}{z}\IM(g_{1,5,1}(iq_1^7z^2,iq_1^3/z^2,q_1^2,1/q_1^4z^4,q_1^4z^4))\\
=&zj(qz^4;q^2)\left(m(-q^5z^{12},q^{12},1/qz^4)-\frac{1}{q^2z^8}m(-1/qz^{12},q^{12},qz^4)\right),
\end{align*}
and
\begin{align*}
&\frac{q_1}{z}\IM(\Theta_{1,4}(iq_1^7z^2,iq_1^3/z^2,q_1^2))\\
=&\frac{z^5J_{12}J_4^2j(-qz^4;q^{12})}{qJ_8J_6j(-z^8;q^{12})j(-q^4z^8;q^{12})}\left(\frac{J_{12}^2J_8}{J_{24}^2J_4}j(q^{14}z^8;q^{24})^2-\frac{q^2J_{24}^2J_6J_4^2}{J_{12}^2J_8J_2}j(q^8z^8;q^{12})\right).
\end{align*}
Hence by Theorem \ref{24th1}, \eqref{24mid} is equivalent to
\begin{align}
\label{24mid1}
&\frac{J_1^2J_{12}}{J_6^2J_2}\left(\sum_{k=-\infty}^{\infty}\frac{(-1)^{k-1}(12k-1)q^{6k^2}}{1+q^{12k-1}}+\sum_{k=-\infty}^{\infty}\frac{(-1)^{k-1}(12k-7)q^{6k^2-2}}{1+q^{12k-7}}+\frac{2qJ_{12}^2J_3^2J_2^6}{J_6^2J_4J_1^3}\right)\\
\nonumber
=&g'(1),
\end{align}
where
\begin{align*}
g(z):=&\IM(f(z))\\
=&zj(qz^4;q^2)\left(m(-q^5z^{12},q^{12},1/qz^4)-\frac{1}{q^2z^8}m(-1/qz^{12},q^{12},qz^4)\right)\\
-&\frac{z^5J_{12}J_4^2j(-qz^4;q^{12})}{qJ_8J_6j(-z^8;q^{12})j(-q^4z^8;q^{12})}\left(\frac{J_{12}^2J_8}{J_{24}^2J_4}j(q^{14}z^8;q^{24})^2-\frac{q^2J_{24}^2J_6J_4^2}{J_{12}^2J_8J_2}j(q^8z^8;q^{12})\right).
\end{align*}
By Theorem \ref{mthe33}
\begin{align*}
m(-q^5z^{12},q^{12},1/qz^4)=&m(-q^5z^{12},q^{12},q^6/z^{12})\\
+&\frac{J_{12}^3j(q^5z^8;q^{12})j(-q^2z^4;q^{12})}{j(qz^4;q^{12})j(q^6z^{12};q^{12})j(-q^4z^8;q^{12})j(-q^{11};q^{12})},
\end{align*}
and
\begin{align*}
m(-1/qz^{12},q^{12},qz^4)=&m(-1/qz^{12},q^{12},q^6z^{12})\\
-&\frac{qz^{12}J_{12}^3j(q^5z^8;q^{12})j(-q^6z^4;q^{12})}{j(qz^4;q^{12})j(q^6z^{12};q^{12})j(-z^8;q^{12})j(-q^5;q^{12})}.
\end{align*}
Hence $g(z)=g_1(z)-g_2(z)$ where
$$
g_1(z):=zj(qz^4;q^2)\left(m(-q^5z^{12},q^{12},q^6/z^{12})-\frac{1}{q^2z^8}m(-1/qz^{12},q^{12},q^6z^{12})\right),
$$
and
\begin{align*}
g_2(z):=&\frac{z^5J_{12}J_4^2j(-qz^4;q^{12})}{qJ_8J_6j(-z^8;q^{12})j(-q^4z^8;q^{12})}\left(\frac{J_{12}^2J_8}{J_{24}^2J_4}j(q^{14}z^8;q^{24})^2-\frac{q^2J_{24}^2J_6J_4^2}{J_{12}^2J_8J_2}j(q^8z^8;q^{12})\right)\\
-&\frac{zJ_{12}^3j(qz^4;q^2)j(q^5z^8;q^{12})}{j(qz^4;q^{12})j(q^6z^{12};q^{12})}\left(\frac{j(-q^2z^4;q^{12})}{j(-q^4z^8;q^{12})j(-q^{11};q^{12})}+\frac{z^4j(-q^6z^4;q^{12})}{qj(-z^8;q^{12})j(-q^5;q^{12})}\right).
\end{align*}
Since by \eqref{pj21}
$$
\frac{d}{dz}j(qz;q^2)\bigg|_{z=1}=0,
$$
and  definition of $m(x,q,z)$ we have
$$
g_1'(1)=\frac{J_1^2J_{12}}{J_6^2J_2}\left(\sum_{k=-\infty}^{\infty}\frac{(-1)^{k-1}(12k-1)q^{6k^2}}{1+q^{12k-1}}
+\sum_{k=-\infty}^{\infty}\frac{(-1)^{k-1}(12k-7)q^{6k^2-2}}{1+q^{12k-7}}\right).
$$
By Lemma \ref{24lmt} we have
$$
g_2'(1)=-\frac{2qJ_{12}^3J_3^2J_2^5}{J_6^4J_4J_1}.
$$
Then \eqref{24mid1} holds.
\end{proof}

\eqref{24mid} is very similar to \cite[Eq. (2.1)]{Be-Ch-07}
$$
\frac{J_1^2}{J_2}\phi_{-}(q)=\sum_{1-|n|\leq j\leq |n|}sg(n)(-1)^{n-1}q^{3n^2-n-2j^2+j},
$$
which easily implies that \cite[Lemma 3.8]{Ch-Ga}
$$
\mathscr{F}_{24,-1}(q)\equiv -\phi_{-}(q)\pmod 4.
$$
We find that \eqref{24mid} is also similar  to \cite[Lemma 4.4]{Ch-Ga}
$$
\frac{J_1^2}{J_2}\psi(q)=\sum_{1-|n|\leq j\leq |n|}sg(n)(-1)^{j-1}q^{3n^2-n-2j^2+j}.
$$

\section{Hurwitz class numbers and partitions}

In this section, we start with a formula by Humbert \cite[P. 346]{Hu-07}
\begin{align}
\label{smain}
\mathscr{F}_{4,-1}(q)=&\sum_{m=0}^{\infty}\sum_{u=-m}^{m}\frac{q^{(2m+1)^2/4+(2m+1)/2+1/4-u^2}}{1-q^{2m+1}}\\
\nonumber
=&\sum_{m=0}^{\infty}\sum_{u=-m}^{m}\frac{q^{(m+1)^2-u^2}}{1-q^{2m+1}}.
\end{align}
Then we will prove the combinatorial interpretations of $F(4n-1)$ and $H(8n-1)$  which appeared on OEIS (see A238872 and A321440).

Denote $P(n)$ to be the number of strongly unimodal compositions of $n$ with absolute difference of successive parts equal to 1(see A238872 on OEIS). It is easy to see that
$$
n=x+(x+1)+\cdots+(y-1)+y+(y-1)+\cdots+(z+1)+z=y^2-\frac{x(x-1)}{2}-\frac{z(z-1)}{2}.
$$
We find that
\beq
\label{s1id}
\sum_{n=1}^{\infty}P(n)q^n=\sum_{y=1}^{\infty}\sum_{x=1}^{y}\sum_{z=1}^{y}q^{y^2-x(x-1)/2-z(z-1)/2}.
\eeq

\begin{theorem}
For all $n\in \mathbb{N^*}$, we have
$$
P(n)=F(4n-1).
$$
\end{theorem}

\begin{proof}
Let
$$
A(m,u):=\frac{q^{(m+1)^2-u^2}}{1-q^{2m+1}},
$$
and
$$
B(y,x,z):=q^{y^2-x(x-1)/2-z(z-1)/2}.
$$
Then
\begin{align*}
A(m,u)=&\frac{q^{(m+1)^2-u^2}}{1-q^{2m+1}}=\sum_{k=1}^{\infty} q^{m^2-u^2+k(2m+1)}=\sum_{k=1}^{\infty} q^{(k+m)^2-(k+u)(k+u-1)/2-(k-u)(k-u-1)/2}\\
=&\sum_{k=1}^{\infty} B(k+m,k+u,k-u).
\end{align*}
Hence by \eqref{smain} and \eqref{s1id} and noting that
$$
B(y,x,z)=B(y,x,1-z)=B(y,1-x,z),
$$
we have
\begin{align*}
\sum_{n=1}^{\infty}F(4n-1)q^n=&\sum_{m=0}^{\infty}\sum_{u=-m}^{m}A(m,u)=\sum_{m=0}^{\infty}\sum_{u=-m}^{m}\sum_{k=1}^{\infty} B(k+m,k+u,k-u)\\
=&\sum_{y=1}^{\infty}\sum_{(x,z)\in D_y}B(y,x,z)=\sum_{y=1}^{\infty}\sum_{x=1}^{y}\sum_{z=1}^{y}B(y,x,z)=\sum_{n=1}^{\infty}P(n)q^n,
\end{align*}
where
$$
D_y:=\{(x,z):x\leq y, z\leq y, x+z\geq 2, x\equiv z\pmod 2\}.
$$
\end{proof}

Denote $Q(n)$ to  be the number of partitions of $n$ into consecutive parts, all singletons except the largest(see A321440 on OEIS). It is easy to see that
$$
n=(m+1)+(m+2)+\cdots+(l-1)+l+l+\cdots+l=l(l+1)/2-m(m-1)/2+kl
$$
Then
\begin{align}
\label{s2id}
\sum_{n=1}^{\infty}Q(n)q^n=&\sum_{l=1}^{\infty}\sum_{m=1}^{l}\sum_{k=0}^{\infty}q^{l(l+1)/2-m(m-1)/2+kl}\\
\nonumber
=&\sum_{l=1}^{\infty}\sum_{m=1}^{l}\frac{q^{l(l+1)/2-m(m-1)/2}}{1-q^l}.
\end{align}

\begin{theorem}
For all $n\in \mathbb{N^*}$, we have
$$
Q(n)=H(8n-1).
$$
\end{theorem}

\begin{proof}
It is easy to see that
$$
\sum_{l=1}^{\infty}\sum_{m=1}^{l}\frac{q^{l(l+1)/2-m(m-1)/2}}{1-q^l}=\sum_{D}\frac{q^{l(l+1)/2-m(m-1)/2}}{1-q^l},
$$
where
$$
D=\{(l,m):1-l\leq m\leq l,m\equiv l+1\pmod 2\}.
$$
So replacing $l$, $m$ by $s+t+1$ and $s-t$ respectively in \eqref{s2id} we have
\begin{align*}
\sum_{n=1}^{\infty}Q(n)q^n=&q\sum_{s=0}^{\infty}\sum_{t=0}^{\infty}\frac{q^{2st+2s+t}}{1-q^{s+t+1}}=q\sum_{s=0}^{\infty}\sum_{t=0}^{\infty}\frac{q^{2st+2s+t}}{1-q^{2s+2t+2}}(1+q^{s+t+1})\\
=&q\sum_{sg(s)=sg(t)}sg(s)\frac{q^{2st+2s+t}}{1-q^{2(s+t)+2}}.
\end{align*}
Replacing $m$, $u$ by $s+t$ and $s-t$ respectively on the right-hand side of \eqref{smain}, we have
\beq
\label{s2id1}
\sum_{D_0}\frac{q^{(m+1)^2-u^2}}{1-q^{2m+1}}=\sum_{s=0}^{\infty}\sum_{t=0}^{\infty}\frac{q^{4st+2s+2t+1}}{1-q^{4(s+t)+2}}+\sum_{s=0}^{\infty}\sum_{t=0}^{\infty}\frac{q^{4st+4s+4t+2}}{1-q^{4(s+t)+2}}.
\eeq
Replacing $m$, $u$ by $s+t+1$ and $s-t$ respectively on the right-hand side of \eqref{smain}, we have
\beq
\label{s2id2}
\sum_{D_1}\frac{q^{(m+1)^2-u^2}}{1-q^{2m+1}}=\sum_{s=0}^{\infty}\sum_{t=0}^{\infty}\frac{q^{4st+4s+4t+4}}{1-q^{4(s+t)+6}}+\sum_{s=0}^{\infty}\sum_{t=0}^{\infty}\frac{q^{4st+6s+6t+7}}{1-q^{4(s+t)+6}},
\eeq
where
$$
D_i=\{(m,u):-m\leq u\leq m,m-u\equiv i\pmod 2\}.
$$
By \eqref{smain}, \eqref{s2id1} and \eqref{s2id2} we have
\begin{align*}
\mathscr{F}_{8,-1}(q)=&\sum_{s=0}^{\infty}\sum_{t=0}^{\infty}\frac{q^{2st+2s+2t+1}}{1-q^{2(s+t)+1}}+\sum_{s=0}^{\infty}\sum_{t=0}^{\infty}\frac{q^{2st+2s+2t+2}}{1-q^{2(s+t)+3}}\\
=&q\sum_{sg(s)=sg(t)}sg(s)\frac{q^{2st+2s+2t}}{1-q^{2(s+t)+1}}.
\end{align*}
Hence
$$
\mathscr{F}_{8,-1}(q)=\sum_{n=1}^{\infty}Q(n)q^n,
$$
since
\begin{align*}
&q\sum_{sg(s)=sg(t)}sg(s)\frac{q^{2st+2s+2t}}{1-q^{2(s+t)+1}}=q\sum_{sg(s)=sg(t)=sg(r)}q^{2(st+sr+tr)+2s+2t+r}\\
=&q\sum_{sg(s)=sg(t)=sg(r)}q^{2(st+sr+tr)+2s+t+2r}=q\sum_{sg(s)=sg(t)}sg(s)\frac{q^{2st+2s+t}}{1-q^{2(s+t)+2}}.
\end{align*}
\end{proof}

\begin{cor}
For all $n\in \mathbb{N^*}$, we have
$$
Q(n)=P(2n).
$$
\end{cor}

\subsection*{Acknowledgements}
The first author was supported by Shanghai Sailing Program (21YF1413600). The second author was supported by the fellowship of China Postdoctoral Science Foundation (2022M712422).

\end{document}